\documentclass[12pt,psamsfonts,leqno,oneside,letterpaper,left=1.25truein, top=1truein, margin=2.5cm ]{amsart}%{article}
\usepackage{indentfirst}
\usepackage[below]{placeins} % use \FloatBarrier in the b1ody

\usepackage[colorlinks,linkcolor=blue,citecolor=blue, pdfauthor = Huygens\ C.\ Ravelomanana,pdfsubject=Low\ dimensional\ topology]{hyperref}

\usepackage{amsmath,amssymb,amsfonts,amsthm,latexsym}
\usepackage{amsfonts}
\usepackage[utf8]{inputenc} \usepackage{ae,aecompl}     % changes to Type 1 fonts
\usepackage{amscd}
\usepackage{mathrsfs}       % fonts for math script (\mathscr)
\usepackage{fancyhdr}
\usepackage{graphicx,psfrag}
\usepackage{epsf} 
\usepackage[all]{xy}
\usepackage{tikz}
\usetikzlibrary{chains,scopes,shapes,arrows,trees,matrix,positioning,fit,decorations.pathmorphing}

\usepackage[nofoot,letterpaper,left=1.75truein, top=1truein]{geometry}

\usepackage{multicol}

\theoremstyle{definition}

\newtheorem{thm}{Theorem}[section]% or [chapter] or [subsection]
\newtheorem{innercustomcor}{Corollary}
\newenvironment{customcor}[1]
  {\renewcommand\theinnercustomcor{#1}\innercustomcor}
  {\endinnercustomcor}

\newtheorem{lem}[thm]{Lemma}
\newtheorem{conj}[thm]{Conjecture}
\newtheorem{cor}[thm]{Corollary}
\newtheorem{pro}[thm]{Proposition}
\newtheorem{defn}[thm]{Definition}

\theoremstyle{definition}%%%%%%%%%%%%%%%%%%%%%%%%%%%%%%%%%%

\newtheorem{exa}[thm]{Example}
\newcommand {\C}{\mathbb C}
\newcommand {\Z}{\mathbb Z}

\newcommand {\Q}{\mathbb Q}

\newcommand {\Ta}{\mathbb{ T}_{\alpha}}
\newcommand {\Tb}{\mathbb{ T}_{\beta}}

\newcommand {\nb}{\mathcal N}

\newcommand{\hfinfty}{HF^{\infty}}
\newcommand{\hfm}{HF^-}
\newcommand{\hfp}{HF^+}
\newcommand{\hfred}{HF_{\mathrm{red}}}
\newcommand{\hfc}{HF^{\circ}}
\newcommand{\cfkinfty}{CFK^{\infty}}

\newcommand{\hfk}{\widehat{HFK}}

\newcommand{\on}{\operatorname}

\newcommand{\Spinc}{\on{Spin}^c}
\newcommand{\gr}{\on{gr}}

\newcommand{\x}{{\bf x}}
\newcommand{\y}{{\bf y}}

\newcommand{\s}{\mathfrak{s}}

\newcommand\liftGr{\widetilde\gr}

\newcommand{\rank}{\text{rank}}

\newcommand{\spinc}{\mathrm{ Spin}^c }

\title{Exceptional Cosmetic surgeries on $S^3$}
\author{Huygens C. Ravelomanana}
\date{}

\begin{document}

\maketitle
\begin{abstract}
%\begin{center}
%\begin{minipage}{8cm}
%\begin{center} \textbf{Abstract} \end{center}
This paper concerns the truly or purely cosmetic surgery conjecture. We give a survey on exceptional surgeries and cosmetic surgeries. We prove that the slope of an exceptional truly cosmetic surgery on a hyperbolic knot in $S^3$ must be $\pm 1$ and the surgery must be toroidal but not Seifert fibred. As consequence we show that there are no  exceptional truly cosmetic surgeries on certain types of hyperbolic knot in $S^3$. We also give some properties of Heegaard Floer correction terms and torsion invariants for exceptional cosmetic surgeries on $S^3$.
%\end{minipage}
%\end{center}
\end{abstract}

\section{introduction}

Dehn surgery is an essential tool in 3-manifold topology. \textit{Cosmetic surgery} addresses the question: when do two surgeries along the same knot, but with distinct slopes, produce the same manifold? Such a situation does happen, but observations suggest that for generic knots and 3-manifolds this should be very rare. Let $Y$ be an oriented 3-manifold and $K$ a knot in $Y$. Let $Y_K(\alpha)$ and $Y_K(\beta)$ be the result of two Dehn filling on $K$ with two distinct slopes $\alpha$ and $\beta$. If $Y_K(\alpha) \cong Y_K(\beta)$ as oriented manifold we say that the two surgeries are \textit{truly (or purely) cosmetic}.
The following conjecture was proposed by Gordon in [\cite{Gordon-Kyoto}, Conjecture 6.1]  and  is stated as  conjecture (A) in problem 1.81 of Kirby list of problems in low-dimensional topology \cite{Kirby-list}.

\begin{conj}[{\bf Cosmetic surgery conjecture}]
Let $M$ be a compact connected  oriented irreducible 3-manifold with torus boundary and which is not a solid torus. Let $\alpha$ and $\beta$ be two inequivalent slopes on $\partial M$. If $M(\alpha)\cong M(\beta)$, then the homeomorphism is orientation-reversing. Equivalently, two surgeries on inequivalent slopes are never truly cosmetic.
\end{conj}
Here two slopes on $\partial M$ are called equivalent if there exists 
an orientation-preserving homeomorphism of $M$ which takes one to the other.

For the trivial knot, it is known that there can be infinitely many distinct surgeries which give the same output. 
Mathieu, in \cite{Mathieu}, gives an infinite family of distinct Dehn surgeries on a trefoil knot in $S^3$ which give homeomorphic manifolds. 
In \cite{Weeks} Bleiler, Hodgson and Weeks described an oriented  hyperbolic $3$-manifold with torus boundary having two distinct \textit{Dehn fillings} which give two oppositely oriented copies of the lens space $L(49, 18)$. 

In  \cite{Boyer-Lines}  Boyer and Lines proved that $\Delta_K''(1)$ must vanish in order to have truly cosmetic surgery.

\begin{pro}[Boyer and Lines]\label{my-proposition-1} Let $K$ be a  knot in an $\Z$-homology sphere $Y$. If $\Delta''_K(1)\neq 0$,  then there is no orientation preserving homeomorphism between  $Y_K(r)$ and $Y_K(s)$ if $r\neq s$.
\end{pro}

Recently, with help of Heegaard Floer theory and Casson invariant, new criteria for cosmetic surgeries on knots in $S^3$, and more generally knots in $L$-space homology spheres, have been established by Yi and Zhongtao Wu in \cite{Ni-Wu}. 

\begin{thm}[Yi Ni and Zhongtao Wu]\label{Ni-Wu-theorem}
Suppose $K$ is a nontrivial knot in $S^3$, $r,s\in \Q \cup\left\{\infty \right\}$ are two distinct slopes such that $S_K (r)$ is homeomorphic to $S_K(s)$ as oriented manifolds. Then $r,s$ satisfy that
\begin{enumerate}
\item $r=-s$;
\item suppose $r = p/q$, where $p,\ q$ are coprime integers, then: $\  \ q^2 \equiv -1 \ \ \left[\text{mod}\ p\right];$
\item $\tau(K)=0$, where $\tau$ is the concordance invariant defined by Ozsv\'ath-Szab\'o and Rasmussen.
\end{enumerate}
\end{thm}

\bigskip
%We begin our exploration of cosmetic surgery with the case of hyperbolic knots in $S^3$.  
Using Ni and Wu's result combined with the progress made on exceptional surgeries on $S^3$ and the result of Boyer and  Lines we provide the following new property of cosmetic surgery on hyperbolic knot in $S^3$.

\begin{thm}\label{my-theorem-A}
Let $K$ be a hyperbolic knot in $S^3$, and $r,s\in \Q \cup\left\{\infty \right\}$  two distinct exceptional slopes  on $\partial \mathcal{N}(K)$. If $S_K (r)$ is homeomorphic to $S_K(s)$ as oriented manifolds, then the surgery must be   toroidal and non-Seifert fibred, moreover $\{r,s\}=\{+1,-1\}$.
\end{thm} 

From this we deduce that there are no exceptional truly cosmetic  surgeries on some classical families of knots.

\begin{cor}\label{my-corollary-2-chap3}
There are no truly cosmetic surgeries on non-trivial algebraic knot in $S^3$.
\end{cor}

\begin{cor}\label{my-corollary-3}
There are no exceptional truly cosmetic surgeries on an alternating hyperbolic knot in $S^3$.
\end{cor}

\begin{cor}\label{my-corollary-4}
There are no exceptional truly cosmetic surgeries on arborescent knots in $S^3$.
\end{cor}

We have also the following properties of the Heegaard Floer correction term and the Alexander polynomial.

\begin{customcor}{\ref{my-corollary-A2}}
If a hyperbolic knot $K\subset S^3$ admits an exceptional truly cosmetic surgery then the Heegaard Floer correction term of any $1/n$ ($n\in \Z$) surgery on $K$ satisfies $$d(S^3_K(1/n))=0.$$
\end{customcor}

\begin{customcor}{\ref{my-corollary-A1}}
If a 3-manifold $Y$ is the result of an exceptional truly cosmetic surgery on a hyperbolic knot $K$ in $S^3$ then:
$$|t_0(K)| + 2 \sum_{i=1}^n |t_i(K)|\leq \rank  \hfred(Y),$$
where the number $t_i(K)$ for $i\in \Z$ is the torsion invariant of the Alexander polynomial  $\Delta_K(T)$ of $K$ and $n$ is the degree of $\Delta_K(T)$.
\end{customcor}

\subsection*{Organization.}	
This paper is organized as follow.  In section~\ref{Exceptional cosmetic surgeries} we  give some background and survey exceptional surgeries, cosmetic surgeries. In section~\ref{Cosmetic Surgery on $S^3$} we survey some result on cosmetic surgeries obtained using Heegaard Floer theory and give the proof of Theorem~\ref{my-theorem-A}. In section~\ref{special examples} we give examples of family of knots in $S^3$ which do not admit exceptional truly cosmetic surgeries. In section~\ref{miscellaneous properties of HFH} we enumerate some miscellaneous properties of Heegaard Floer invariant for exceptional truly cosmetic surgeries on $S^3$.

\subsection*{Acknowledgment.}
I would like to thank my supervisor Steven Boyer for his support and for suggesting
 the cosmetic surgery problem. This work was carried out while the  author was a graduate student at UQ\`AM and CIRGET in Montr\'eal.

\section{Exceptional cosmetic  surgeries}\label{Exceptional cosmetic  surgeries}
\subsection{Topological background}
\subsubsection{Distance between slopes}
Let $M$ be a compact, connected, oriented 3-manifold and let $T \subset \partial M$ be a torus. 
%\paragraph{Distance between two slopes.}
%\textbf{Distance between two slopes.}
The distance, denoted $\Delta(\alpha, \beta)$, between two slopes $\alpha$ and $\beta$ on $T$ is their minimal geometric intersection number. That is
$$\Delta(\alpha, \beta) =\min \left\lbrace  \sharp C_1\cap C_2\  : \ C_1, C_2 \ \ \textnormal{simple closed curve representing}\ \ \alpha \ \textnormal{and}\  \beta \ \ \textnormal{respectively}\right\rbrace$$
The distance has the following straightforward properties:
\begin{itemize}
\item[•]  $\Delta(\alpha, \beta) = |\alpha \cdot \beta|.$
\item[•]  $\Delta(\alpha, \beta) =0$ iff $\alpha = \beta $.
\item[•]  $\Delta(\alpha, \beta) =1$ iff $\left\lbrace \alpha, \beta \right\rbrace$ form a basis of $H_1(\partial M;\Z)$.
\item[•]  If we fix a basis $\left\lbrace \mu, \lambda \right\rbrace$ of $H_1(T;\Z)$, then for $\alpha = p\mu +q \lambda$ and  $\beta = p'\mu +q' \lambda$ 
$$\Delta(\alpha, \beta) = |pq'-qp'|.$$
\end{itemize}

When $ \partial M$ consist of a single torus there is a formula relating the order of the first homology of the filled manifold to the distance of the filling slope from the rational longitude.
\begin{lem}[Watson \cite{Watson-PhD}] \label{watson-lemma}
Let $\alpha$ be  a slope on $ \partial M$. There is a constant $c_M$ such that
$$|H_1(M(\alpha);\Z)|= c_M\ \Delta(\alpha,\lambda_M).$$
\end{lem}
If we denote $i:\partial M\to M$ the natural inclusion then the constant $c_M$ is the quantity
$$c_M=\left|\text{Tor}(H_1(M;\Z))\right|\ \text{ord}(i_*\lambda_M),$$
where $\text{ord}(i_*\lambda_M)$ is the order of $i_*\lambda_M$ in the homology of $M$.

%Misy olana ilay surgery on link rehefa general closed 3-manifold, mila mampiditra ilay primitive homology class sns.....

\bigskip

% Jereo boky Invariants for Homology 3-Spheres, Nikolai Saveliev, takelaka faha 31
\subsubsection{Surgery on a link.} Assume that $Y$ is an integer homology sphere. Let $L=K_1 \cup \cdots \cup K_m$ be a link in $Y$. Each component of $L$ has a canonical longitude therefore every surgery on $L$ can be described by an $m$-tuple $(p_1/q_1,\cdots,p_m/q_m)$ of elements in $\Q \cup \{\infty\}$. By a \textit{framed link} we mean the data of the link $L$ with such an $m$-tuple. The $m$-tuple itself will be called the \textit{framing} of the link. A framed link will be denoted by calligraphic letter, like $\mathcal{L}$. We will write $Y(\mathcal{L})$ for the result of a Dehn surgery on a framed link $\mathcal{L}$. The \textit{framing matrix} of a framed link $\mathcal{L}$ in $Y$ is the matrix $F(\mathcal{L})$ defined by
$$ F(\mathcal{L})_{ij} =  \left\lbrace\begin{array}{l}
         q_j\ \text{lk}(K_i,K_j) \ \ \ \textnormal{if } \ \ \ i\neq j \\
          p_i \ \ \ \textnormal{if } \ \ \ i= j 
         \end{array}\right.
$$
where lk$(,)$ denotes the linking number. The framing matrix gives a presentation for $H_1(Y(\mathcal{L}),\Z)$, in particular
$$\left| \det\left(F(\mathcal{L})\right) \right|= \left| H_1(Y(\mathcal{L});\Z) \right|.$$
For the case of a 2-component link, the framing matrix has the form
$$
F(\mathcal{L})=
\begin{pmatrix}
p_1 & q_2\ \text{lk}(K_1,K_2)\\
q_1\ \text{lk}(K_2,K_1) & p_2  
\end{pmatrix}
$$

For more details we refer to \cite{Saveliev}.
% framed knot ihany koa tokony ho resahana kely

\subsection{Exceptional surgeries}
% asiana definition essential surfaces
A compact, connected orientable 3-manifold $M$ will be called \textit{irreducible} if every properly embedded 2-sphere in $M$ bounds a 3-ball. Otherwise $M$ will be called \textit{reducible}. It will be called \textit{boundary irreducible} if every simple closed curve on $\partial M$  which bounds a disk in $M$ bounds a disk in $\partial M$, and otherwise \textit{boundary reducible}. All embedded surfaces in a 3-manifold we will be considering will be bicollared if not stated otherwise. From now on we will use the following definition.
\begin{defn}\label{definition:essential surface}
A properly embedded non-empty surface $F$ in a compact,  orientable 3-manifold $M$ is said to be \textit{essential} if it is a 2-sphere which does not bound a 3-ball or if it has the following properties:
$$\begin{aligned}
& \text{1. $F$ has no 2-sphere component,}\\
& \text{2. the inclusion morphism $\pi_1(F_i)\to \pi_1(M)$ is injective for every component $F_i$ of $F$,}\\
& \text{3. no component of $F$ is parallel into $\partial M$.}
\end{aligned}
$$
\end{defn}

Let $F \subset M$ be a properly embedded surface with boundary and $T$ be a torus component of $\partial M$. Each component of $\partial F \cap T$ is a simple closed curve on $T$ and they all determine the same slope. A slope $r$ on $T$ is called \textit{boundary slope} if it is the slope of a boundary component of an essential surface in $M$. If the corresponding surface is a punctured torus then the slope will also be called a \textit{toroidal slope}.

If all the components of $\partial M$ are tori or $\partial M$ is empty, $M$ is said to be \textit{hyperbolic} if its interior admits a complete finite volume Riemannian metric of constant sectional curvature $-1$. If $M$ is hyperbolic then it is irreducible, boundary irreducible, contains no essential tori or annuli and is not a Seifert fibred manifold. Thurston's hyperbolization theorem implies that the last statement is an equivalence. A hyperbolic structure on $M$ is unique up to isometry by the Mostow-Prasad rigidity theorem. 

Fix $M$ a hyperbolic 3-manifold  with $\partial M$ a union of tori. In this section we will discuss Dehn filling of $M$. Let $T$ be a component of $\partial M$. By studying metric completions of incomplete ``hyperbolic" 3-manifolds, W. Thurston discovered that except for a finite number of slopes all the Dehn fillings of $M$ along $T$ give hyperbolic manifolds.
\begin{thm}[Thurston, \cite{Thurston}]\label{Thurston-Theorem}
Let $M$ be a compact connected oriented 3-manifold with boundary a union of tori. Let $T$ be a component of $\partial M$. If int$(M)$ admits a complete finite volume hyperbolic structure, for all but finitely many slopes $\alpha$ on $T$, $M(\alpha)$ is hyperbolic and the core of the Dehn filling is isotopic to the unique shortest geodesic in this manifold.
\end{thm}

Let's consider the set $E(M,T)$ of non-hyperbolic slope on $T$. A slope in $E(M,T)$ is called an \textit{exceptional slope}. By the above theorem it is a finite set, and one goal of Dehn filling theory is to understand this set of slopes. One of the main ``techniques" in this study is to find a bound on the distance $\Delta(r,s) $ between two exceptional slopes $r$ and $s$.

\begin{thm}[Lackenby-Meyerhoff, \cite{Lackenby-Meyerhoff}] \label{Lackenby-Meyerhoff}
Let $M$ be a compact orientable
3-manifold with boundary a torus, and with interior admitting
a complete finite-volume hyperbolic structure. If $r$ and
$s$ are exceptional slopes on $\partial M$, then their
intersection number $\Delta(r, s)$ is at most $8$.
\end{thm}

This bounds is achieved by the figure-8 exterior, indeed 
$$E(\text{figure-8 exterior})=\{\infty, 0,\pm 1,\pm 2,\pm 3, \pm 4\}.$$

%Atao ohatra ilay Whitehead link by $W$.

It was conjectured by Gordon that the distance of two exceptional slopes is less than $5$ for almost all hyperbolic 3-manifold with torus boundary.

\begin{conj}
Let $M$ be an hyperbolic 3-manifold with boundary a torus. If $\alpha$ and $\beta$ are two exceptional slopes on $\partial M$, then $\Delta(\alpha,\beta) \leq 5$ unless $M$ is one of $W(1), W(2), W(-5/2),$ or $ W(-5)$, see $\textnormal{figure 1}$.
\end{conj}
\begin{figure}[h!]
\begin{center}
\begin{minipage}{11cm}
\includegraphics[width=110mm]{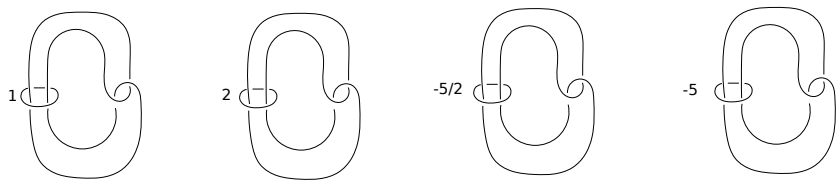}
\caption{ $W(1), W(2), W(-5/2), W(-5)$}
\end{minipage}
\end{center}
\end{figure}

\bigskip

The conjecture is known to be true if the two slopes are both toroidal \cite{punctured-tori}. 

For non-toroidal exceptional surgeries there are three principal results.

\begin{thm} [Cyclic surgery theorem, Culler-Gordon-Luecke-Shalen\cite{Cyclic-paper}]
Let $M$ be a compact, oriented, irreducible 3-manifold which is not a Seifert fibred space. Assume that $\partial M$ is a torus and let $r,s$ be two slopes on $\partial M$. If $\pi_1(M(r))$  and $\pi_1(M(s))$ are cyclic, then $\Delta(r,s)\leq 1$. %Hence there are at most two slopes $r$ such that $pi_1(M(r))$  is cyclic.
\end{thm}

\begin{thm} [Finite surgery theorem, Boyer-Zhang \cite{finite-paper}]
Let $M$ be a compact orientable hyperbolic 3-manifold with torus boundary. If $r,s$ are two slopes on $\partial M$ such that $\pi_1(M(r))$  and $\pi_1(M(s))$ are finite, then $\Delta(r,s)\leq 3$. %Hence there are at most two slopes $r$ such that $pi_1(M(r))$  is cyclic.
\end{thm}

%\bibitem{finite-paper}
% S. Boyer and Xingru Zhang, 
%\newblock{\em A proof of the finite filling conjecture}, 
%\newblock{J. Differential Geom. 59 (2001), no. 1, 87–176}. 

\begin{thm}[Gordon-Luecke,  \cite{Gordon-Luecke-3}]\label{thm:reducible surgery}
Let $M$ be a compact orientable irreducible 3-manifold with torus boundary. If $r,s$ are two slopes on $\partial M$ such that $M(r)$  and $M(s)$ are both reducible, then $\Delta(r,s)\leq 1$. %Hence there are at most two slopes 
\end{thm}

We summarize all the results about the bounds on $\Delta (r,s)$ for $r,s \in E(M)$  in table \ref{distance-table}. 
%We call a slope $r \in E(M)$:
%$$ \begin{aligned}
%  & \bullet \text{{\em reducible}, if $M(r)$ is reducible, \hspace{3cm}}\\
%  & \bullet  \text{{\em toroidal}, if  $M(r)$ contains an essential torus,}\\
%  & \bullet  \text{{\em cyclic}, if   $\pi_1 M(r)$ is cyclic,  {\em finite}, if   $\pi_1 M(r)$ is finite but not cyclic,}\\
%  & \bullet \text{{\em small Seifert}, if $M(r)$ is a small seifert manifold.}
%  \end{aligned}
%$$
%
\begin{table}[h!] \label{distance-table}
\begin{center}
\caption{Distance table.}
\end{center}
\begin{center}
\begin{tabular}{|c|c|c|c|c|c|} 
\hline 
\rule[-1ex]{0pt}{2.5ex}  & reducible & cyclic & finite & toroidal & small Seifert\\ 
\hline 
\rule[-1ex]{0pt}{2.5ex} reducible & 1 & 1 & 1 & 3 & 4\\ 
\hline 
\rule[-1ex]{0pt}{2.5ex} cyclic &   & 1 & 2 &  8 & 8\\ 
\hline 
\rule[-1ex]{0pt}{2.5ex} finite &   &   & 3 &  8 & 8\\ 
\hline 
\rule[-1ex]{0pt}{2.5ex} toroidal &   &   &   & 8  & 8\\ 
\hline 
\rule[-1ex]{0pt}{2.5ex} small Seifert &   &   &   &  & 8 \\ 
\hline 
\end{tabular} 
\end{center}
\end{table}

\bigskip

The list of knots in $S^3$ which admit  pair of toroidal slopes at  distance $4$ or more is also known by work of Gordon \cite{punctured-tori} and Gordon and Ying-Qing Wu in \cite{toroidal}.

%%%%%%%%%Toroidal%%%%%%%%%%%%
%
%%Theorem 24.4 
\begin{thm}[Gordon and Ying-Qing Wu] \label{toroidal:theorem 24.4}
A knot $K$ in $S^3$ is hyperbolic and admits two toroidal surgeries
$S^3_K(r_1 )$, $S^3_K(r_2 )$ with $\Delta(r_1,r_2) \geq 4$  if and only if $(K, r_1 , r_2 )$ is equivalent to one of the following, where $n$ is an integer.

(1) $K = L_1(n)$, $r_1 = 0$, $r_2 = 4$.

(2) $K = L_2(n)$, $r_1 = 2-9n$, $r_2 = -2-9n$.

(3) $K = L_3(n)$, $r_1 = -9 - 25n$, $r_2 = -(13/2) - 25n$.

(4) $K$ is the Figure 8 knot, $r_1 = 4$, $r_2 = -4$.
\end{thm}

The knots $L_1(n), \ L_2(n)$ and $L_3(n)$ are the knots obtained from the right components of the links $L_1,\ L_2,\ L_3$ in Figure 2 after  $1/n$-surgery on the left components. In the particular case where $\Delta(r_1,r_2) = 4$, then  $K = L_1(n)$, $r_1 = 0$, $r_2 = 4$; or $K = L_2(n)$, $r_1 = 2-9n$, $r_2 = -2-9n$.
\begin{figure}[h!]
\begin{center}
\begin{minipage}{11cm}
\includegraphics[width=110mm]{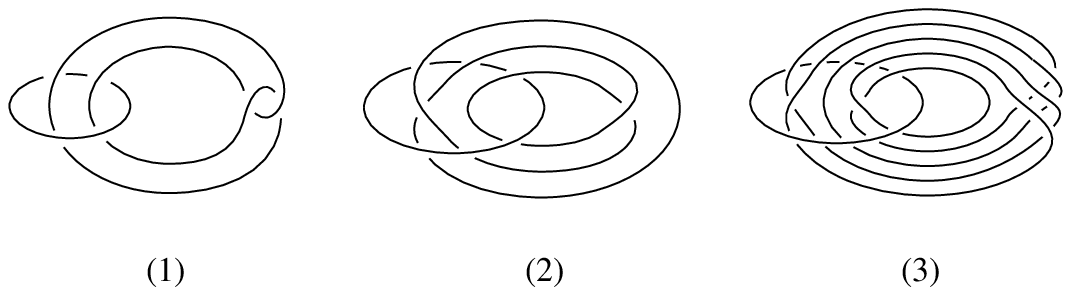}
%\caption{ $L_1, L_2, L_3$}
\end{minipage}
\end{center}
\end{figure}

In term of the slope of the toroidal surgery there is a bound on the denominator $q$ of a toroidal slope $p/q$.

\begin{thm}[Gordon-Luecke, \cite{Gordon-Luecke-4}]\label{denominator-2} %  na 5
Let $K$ be a hyperbolic knot in $S^3$ and  suppose that $S^3_K(p/q)$ contains an essential torus. Then $|q| \leq 2$.
\end{thm}

%%%%%%%%%%%%%%%%%%%%%%%%%%%Non-integral toroidal Dehn surgeries%%%%%%%%%%%%%%%%%%%%%%%%%%

For the case where the slope is non-integral we have a complete understanding of toroidal surgeries which is given by the following theorem.

\begin{thm}[Gordon and Luecke, \cite{Gordon-Luecke-2}]\label{non-integral:theorem 1.1}
 Let $K$ be a hyperbolic knot in $S^3$ that admits a non-integral surgery containing an incompressible torus. Then $K$ is one of the Eudave-Mu\~noz knots $k(l , m, n, p)$
and the surgery is the corresponding half-integral surgery.
\end{thm}
%%%%%%%%%%%%%%%%%%%%%%%%%%%%%%%%%%%%%%%%%%%%%%%%%%%%%%%%%%%%%%%%%%%%%%%%%%%%%%%%%%%%%%%%%%

%\newpage
\subsection{Cosmetic surgery}

\begin{defn}  
Two Dehn fillings $M(\alpha)$ and $M(\beta)$, where $\alpha \ne \beta$, are called cosmetic if there is a homeomorphism $h:M(\alpha)\to M(\beta)$. They are called truly cosmetic if $h$ can be chosen to be orientation-preserving.
We also call two Dehn surgeries cosmetic (resp. truly cosmetic) if the corresponding Dehn fillings are cosmetic (resp. truly cosmetic).
\end{defn}

\begin{exa}
Here are some examples of cosmetic filling for two distinct slopes.
\begin{itemize}
\item[•] If $K$ is an amphicheiral knot in $S^3$ and $M=S^3\setminus \nb( K)$, then $M(\alpha)$ is orientation reversing homeomorphic to $M(-\alpha)$.

\item[•]It was shown by Mathieu \cite{Mathieu} that if $M$ is the complement of the trefoil knot in $S^3$ then we have an infinite family of pairs of distinct slopes which give homeomorphic manifolds. Precisely, for any positive integer $k$,
$$M\left(\frac{18k+9}{3k+1}\right)\cong -M\left(\frac{18k+9}{3k+2}\right).$$
These Dehn filling manifolds are Seifert fibred with normalized Seifert invariants \\ $(0;k-3/2;(2,1),(3,1),(3,2))$. Such manifolds do not admit orientation-reversing homeomorphisms. Therefore the fillings are not truly cosmetic.

\item[•]  If $M$ is the complement of the unknot in $S^3$, $M$ is a solid torus, then the Dehn filling manifolds are \textit{lens spaces} and 
$$
M\left(p / q_1\right) \cong+ M\left(p/q_2\right) \ \ \ \text{iff \ \ $q_2 \equiv q_1^{\pm 1}$ [mod] $p$},
$$
for pairs of relatively prime integers $(p,q_1)$ and $(p,q_2)$.
\end{itemize}
\end{exa}

For the first and the third examples one can find a homeomorphism of $M$ which takes one slope to the other. 

%By this theorem, if the manifold $M$ admit truly cosmetic filling, then $M$ has to be a knot complement in a 3-manifold $Y$ distinct from $S^3$ and $S^2\times S^1$. 
For the case $b_1(Y)>0$ and the core of the Dehn filling is homotopically trivial in $Y$ the following result was proved by Lackenby.

\begin{thm}[Lackenby, \cite{Lackenby}] Let $Y$ be a compact oriented 3-manifold with $H_1(Y,\Q)\neq 0$. Let $K$ be a homotopically trivial knot in $Y$, such that $M=Y\setminus \nb(K)$ is irreducible and atoroidal. Let $M(p/q)$ be the Dehn filling along $K$ with slope $p/q$. Then there is a natural number $C(Y,K)$ which depends only on $Y$ and $K$ such that, if $|q| > C(Y,K)$ then
 $M(p/q)\  \textnormal{ \textit{ is orientation-preserving homeomorphic to}} \ M(p'/q')\  \textnormal{iff}\    p/q=p'/q'.$

\end{thm}
The assumption that $K$ is homotopically trivial can be dropped and replaced by  $K$ homologically trivial and $Y$ reducible or $K$ having finite order in $\pi_1(Y)$ \cite{Lackenby}. Taut sutured manifold theory is used to construct the bound $C(Y,K)$.

\begin{thm}[Wu, \cite{Wu}\label{Wu-theorem}] Let $r$ and $r'$ be two distinct rational numbers with $rr'> 0$, let $K$ be a non-trivial knot in an $L$-space $\Z$-homology sphere $Y$ and let $M=Y\setminus \nb(K)$. Then $M(r)\ncong M(r')$.
\end{thm}

Yi Ni has also studied cosmetic surgeries for manifolds $Y$ with $b_1(Y)>0$. For this he used the Thurston norm with Heegaard Floer homology. 

%W. Thurston, A norm for the homology of 3–manifolds, Mem. Amer. Math. Soc. 59 (1986), no. 339, i–vi and 99–130.
\begin{thm}[Yi Ni \cite{Ni-cosmetic-Thurston-norm}]\label{thm: Cosmetic Thurston Norm}
Suppose $Y$ is a closed $3$--manifold with $b_1(Y)>0$. Let $K$ be
a null-homologous knot in $Y$, so that the inclusion map $Y-K\to Y$
induces an isomorphism $H_2(Y-K)\cong H_2(Y)$ and we can identify
$H_2(Y)$ with $H_2(Y-K)$. Suppose $r\in\mathbb Q\cup\{\infty\}$ and 
let $Y_K(r)$ be the manifold obtained by $r$--surgery on $K$.
Suppose $(Y,K)$ satisfies that
\begin{equation}x_Y(h)<x_{Y-K}(h),\quad \text{for any nonzero
element}\quad h\in H_2(Y). \nonumber
\end{equation} where $x_M$ is the
Thurston norm in $M$. If two rational
numbers $r,s$ satisfy that $Y_{K}(r)\cong\pm Y_{K}(s)$, then
$r=\pm s$.
\end{thm}

We can replace the assumption on the Thurston norm with another condition to obtain 
the following.

\begin{thm}[Yi Ni \cite{Ni-cosmetic-Thurston-norm}]
Suppose $Y$ is a closed $3$--manifold with $b_1(Y)>0$. Suppose $K$
is a null-homologous knot in $Y$. Suppose $x_Y\equiv0$, while the
restriction of $x_{Y-K}$ on $H_2(Y)$ is nonzero. Then we have the
same conclusion as Theorem~\ref{thm: Cosmetic Thurston Norm}. Namely, if two
rational numbers $r,s$ satisfy that $Y_{K}(r)\cong\pm Y_{K}(s)$,
then $r=\pm s$.
\end{thm}

%%%%%%%%%%%%%%%%%%%%%%%%%%%%%%%%%%%%%%%%%%%%%%%%%%%%%%%%%%%%%%%%%%%%%%%%%%%%%%%%%%%%%%%

\bigskip
We will be mainly interested in truly cosmetic surgery along hyperbolic knots $K$ in a rational homology sphere $Y$.  By Theorem \ref{Thurston-Theorem}, $Y_K(r)$ is hyperbolic for all except a finite number of slopes $r$ on $\partial \nb(K)$. Let $r$ and $s$ be such hyperbolic slopes.  Assume $Y_K(r)$ is homeomorphic to $Y_K(s)$. Then by Mostow rigidity there is an isometry $h$  between $Y_K(r)$ and $Y_K(s)$. This isometry takes the unique shortest geodesic in $Y_K(r)$ to the unique shortest geodesic in $Y_K(s)$. Apart from a finite number of slopes, the shortest geodesic  is isotopic to the core of the Dehn filling, and if this is true for the slopes $r$ and $s$ we can assume that $h$ takes the core of the Dehn filling in $Y_K(r)$ to the  core of the Dehn filling $Y_K(s)$. Therefore $h$ takes the meridian $r$ to the meridian $s$. In particular $h$ restricts to a homeomorphism of $Y_K$ which takes $r$ to $s$. Moreover a homeomorphism of a one-cusped orientable hyperbolic 3-manifold which changes the slope of some peripheral curve has to be orientation reversing. Therefore the two slopes $r$ and $s$ are not equivalent.  One can  then deduce the following, see \cite{Weeks}. 

\begin{pro}[Bleiler-Hodgson-Weeks,\cite{Weeks}]
Let $M$ be a compact connected oriented hyperbolic 3-manifold with boundary a torus.  Let $r$ and $s$ be distinct slopes on $\partial M$, such that $M(r)$ (resp. $M(s)$) is hyperbolic and the core of the Dehn filling solid torus is isotopic to the shortest geodesic in $M(r)$ (resp. $M(s)$), which we assume is unique.  If $M(r)$ is homeomorphic to $M(s)$, then there is an orientation-reversing homeomorphism of $M$ which takes $r$ to $s$ but no orientation preserving one. In particular, apart from a finite number of slopes, there are no truly cosmetic filling of $M$ with two inequivalent slopes. 
\end{pro} 

For cosmetic filling on a complete finite volume hyperbolic 3-manifold $M$, the remaining cases  are then:
\begin{itemize}
\item[•] One of the Dehn filling manifolds has a hyperbolic structure but the core of the Dehn filling is not isotopic to the shortest geodesic.
\item[•]  The Dehn filling manifold is not hyperbolic. 
\end{itemize}
The second possibility is the case of exceptional filling. We will focus on this last situation, that is cosmetic surgeries or filling which are also exceptional.

Using Lemma \ref{watson-lemma}, we can deduce the following two preliminary lemmas on cosmetic filling. Let $M$ be a compact, connected, oriented hyperbolic manifold with boundary a torus and assume $b_1(M)=1$. Fix a canonical basis $\left\{\mu,\lambda_M\right\}$  for $H_1(\partial M)$, where $\lambda_M$ is the rational longitude.

\begin{lem}\label{my-lemma-1}
 Let $p/q$ and $p/q'$ be exceptional slopes such that $0<p$ and $q < q'$. If $M(p/q)$ and $M(p/q')$ are homeomorphic then we must be in  one of the following cases:
\begin{multicols}{2}
\begin{itemize}
\item[(a)] $p=1$ and $|q-q'|\leq 8$.
\item[(b)] $p \in \left\{7,5\right\}$ and $q'=q+1$.
\item[(c)] $p\in \left\{4,3\right\}$  and $q'\in \left\{q+1,q+2\right\}$.
\item[(d)] $p=2$ and $q' \in \left\{q+2,q+4\right\}$.
\end{itemize}
\end{multicols}
\end{lem}

\begin{proof}
We have the bound $\Delta(p/q,p/q')=|pq'-qp|=p |q-q'| \leq 8$,  so $p \leq 8$. If $p=1$ then $|q-q'|\leq 8$. If $p \in \left\{8,7,6,5\right\}$ then $|q-q'|\leq 1$ and  $q'=q+1$. On the other hand $p$ and $q$ (resp. $p$ and $q'$) must be  relatively prime, thus since one of $q$ and $q+1$ is even and $p$ cannot be $6$ or $8$.  Similarly if $p\in \left\{4,3\right\}$ then  $|q-q'|\leq 2$ and $q'\in \left\{q+1,q+2\right\}$. If $p=2$ then   $|q-q'|\leq 4$ and $q'\in \left\{q+1,q+2,q+3,q+4\right\}$ but  we must have $q\equiv q'$ $[\text{mod}\ 2]$ so $q'\in \left\{q+2,q+4\right\}$.  
\end{proof}

For the case of reducible or cyclic filling we have the following lemma.

\begin{lem} \label{my-lemma-1bis} Assume the hypothesis of Lemma \ref{my-lemma-1}. If $M(p/q)$ is cyclic or reducible and is homeomorphic to $M(p/q')$ then $p=1$ and $q'=q+1$.
\end{lem}
\begin{proof}
The distance between two reducible slopes or two cyclic slopes is at most one, so $\Delta(p/q,p/q')=|pq'-qp|=p|q'-q| \leq 1$. It follows that $p=1$ and $q'=q+1$.
\end{proof}

 \section{Cosmetic Surgery on $S^3$}\label{Cosmetic Surgery on $S^3$}
%%The theorem does not holds if we replace $S^3$ be an L-space integral homology sphere since it uses the properties od the d-invariant.

%The chapter is structured as follow. In Section \ref{section:Results from Heegaard Floer theory} we review the mapping cone construction in Heegaard Floer homology and then state the rank formula for $\hfhat$. We also prove a corollary of the rank formula and other important results. In Section \ref{section:Proof of Theorem A} we give the proof of Theorem A and its two corollaries. Finally in Section \ref{section:Cosmetic surgery on some special classes of knots} we give examples of families  of knots in $S^3$ which do not admit truly exceptional cosmetic surgery.

\subsection{Results from Heegaard Floer theory}
\label{section:Results from Heegaard Floer theory}

%%%%%%%%%%%% Oct 30 2014%%%%%%%%%%%%%%%%%%%%%%%%%%%%%%%%%%%%%%%%%%%%%%%%%%%%%%

%\subsection{Rational surgery formulas}

Recall that knot Floer homology associates to a null-homologous knot $K$ a 
$\Z\oplus\Z$--filtered $\Z[U]$-complex $\cfkinfty(Y,K)$, generated over $\Z$ by $(\Ta\cap \Tb) \times (\Z\oplus\Z)$ equipped with a function
$\mathcal{F}: (\Ta\cap \Tb) \times (\Z\oplus\Z)\longrightarrow\Z\oplus \Z$ with the property that $\mathcal{F}(U\cdot [\x;i,j])=(i-1,j-1)$ and $\mathcal{F}([\y;i',j'])\leq \mathcal{F}([\x;i,j])$
for all $\y$ having nonzero coefficient in $\partial{\x}$. The Euler characteristic of this homology is also the Alexander polynomial of the knot $K$. From more details on the subject we refer to [\cite{Holo-disks-closed-3mfd}, \cite{Applications}, \cite{Holo-disks-knot}, \cite{Intro-HFH}, \cite{Lectures-HFH}, \cite{Rasmussen-Knot}]

%%%%%%%%%%%%%%%%%%%%%%%%%%%%%%%%%%%%%%%%%%%%%%%%%%%%%%%%%%%%%%%%%%%%%%%%%%%%%%%%

When $K$ is a knot in $S^3$ admitting an $L$-space surgery, the following  characterization of $\hfk(Y,K)$ will be useful. It was proved in [\cite{OSzLensSpace} theorem 1.2]. 
\begin{pro}\label{theorem-HFK-for-L-space-surgery}
Let $K$ be a knot in $S^3$.  If there is a rational number $r$ for which $Y_r(K)$ is an $L$-space, then there is an
increasing sequence of integers
$n_{-k}<...<n_k$
with the property that $n_i=-n_{-i}$, 
and $\hfk(K,j)=0$ unless $j=n_i$ for some $i$, in which case
$\hfk(K,j)\cong \Z$.
\end{pro}

An immediate corollary [\cite{Wu} Corollary 3.8] is a simplified expression for the Alexander polynomials of such knots.

\begin{cor}\label{corollary-Alexander_poly}
Let $K$ be a knot that admits an $L$-space surgery. Then the Alexander polynomial of $K$ has the form
$$\Delta_K(T) = (-1)^k+ \sum_{j=1}^k(-1)^{k-j} (T^{n_j}+T^{-n_j}),$$
for some increasing sequence of positive integers $0<n_1<n_2<...<n_k$.
\end{cor}

%Next we state a result by Rustamov which gives a relation between  Casson-Walker invariant and the ``renormalized Euler characteristic" of the Heegaard Floer homology $\hfhat$.
%\begin{thm}[Rustamov. \cite{Rustamov}]\label{theorem:renormalized-euler-char}
%For any rational homology sphere $M$ we have
%$$ |H_1(M)|\ \lambda(M) =\sum_{\mathfrak{s}\in \text{Spin}^c(M)} \left(\chi(HF_{\mathrm{red}}(M,\s))-\frac{1}{2}d(M,\s)\right)$$
%where $\lambda(M)$ is the Casson-Walker invariant of $M$ and $d$ stands for the correction term in Heegaard Floer homology.
%\end{thm}
%

\bigskip

The following proposition, a variant of a result by Zhongtao Wu and Yi Ni, is  one of the main ingredients for the proof of Theorem \ref{my-theorem-A}. More precisely it implies that the cosmetic surgery cannot be Seifert fibred.
\begin{pro}[Yi Ni and Zhongtao Wu, \cite{Ni-Wu}] \label{Ni-Wu-result}
Let $p,q >0$ be two coprime integers. If there is an orientation preserving homeomorphism  between $S^3_K(p/q)$ and $S^3_K(-p/q)$ then
$$\sum_{ \s \in \spinc (S^3_K(p/q))}  \chi(\hfred (S^3_K(p/q),\s))=0.$$ 
\end{pro}

The next proposition due to P. Ozsv\'ath and Z. Szab\'o will be essential for excluding the possibility of a rational homology 3-sphere Seifert fibred cosmetic surgery.
\begin{pro}[P. Ozsv\'ath and Z. Szab\'o, \cite{Plumbed}] \label{HF-of-seifert-fibred}
Let $Y$ be a rational homology 3-sphere Seifert fibred space. Then for one of the orientations of $Y$, $\hfred(Y)$ is supported in even degree.
\end{pro} 

\begin{cor}\label{my-corollary-1-chap3}
There are no truly cosmetic surgeries on a non-trivial knot in $S^3$ which yields a rational homology sphere Seifert fibred space.
\end{cor}

\begin{proof}
Let $K$ be a non-trivial knot in $S^3$. Let us suppose that there is  an orientation preserving homeomorphism between $S^3_{K}(r)$ and $S^3_{K}(-r)$, by Proposition \ref{Ni-Wu-result} 
 
$$ \sum_{\s\in \Spinc(S^3_{K}(r))}\chi(HF_{\mathrm{red}}(S^3_{K}(r),\s))= 0 .$$
On the other hand by Proposition \ref{HF-of-seifert-fibred}, we can assume $\hfred(S^3_{K}(r))$ is supported in even degree so
$$ \sum_{\s\in \Spinc(S^3_{K}(r))}\chi(HF_{\mathrm{red}}(S^3_{K}(r),\s))= \pm \mathrm{rank}\:HF_{\mathrm{red}}(S^3_{K}(r)).$$
Therefore we must have  $HF_{\mathrm{red}}(S^3_{K}(r)) =0 $ in which case $S^3_{K}(r)$ is an $L$-space. In particular the knot $K$ admit an $L$-space surgery.  Then $K$ admits an integral $L$-space  surgery and by Corollary \ref{corollary-Alexander_poly} the knot Floer homology satisfies:  there is an increasing sequence of integers $n_{-k}<...<n_k$
with the property that $n_i=-n_{-i}$,  and $\hfk(K,j)=0$ unless $j=n_i$ for some $i$, in which case
$\hfk(K,j)\cong \Z$. This implies that the Alexander polynomial of $K$ has the form
$$\Delta_K(T) = (-1)^k+ \sum_{j=1}^k(-1)^{k-j} (T^{n_j}+T^{-n_j}),$$
for some increasing sequence of positive integers $0<n_1<n_2<...<n_k$.

If $\Delta_K(T)=1$, then $\hfk(K,0)=\Z$, and $\hfk(K,j)=0$ for any other $j$.  Hence $g(K)=0$  and $K$ is the unknot, %by the fact that knot Floer homology detects the unknot.  
which we have excluded.
% \cite{Ni}

Thus $\Delta_K(T)\neq 1$ and by a straightforward computation 
$$\Delta_K''(1)=2\sum_{j=1}^{k} (-1)^{k-j} n_j^2.$$  
Then the fact that $0<n_1<n_2<...<n_k$ implies  $\Delta_K''(1)\neq 0$. Using Proposition \ref{my-proposition-1}, $K$ does not admit a truly cosmetic surgery.  

\end{proof}

\subsection{Proof of Theorem \ref{my-theorem-A}}\label{section:Proof of Theorem A}

\begin{lem}\label{my-lemma-1-chap3}
Let $K$ be a hyperbolic knot in $S^3$,  and $r,r'\in \Q \cup\left\{\infty \right\}$  two distinct exceptional slopes on $\partial \mathcal{N}(K)$. If $S_K (r)$ is homeomorphic to $S_K(r')$ as oriented manifolds, then $r$ and $r'$ are in the following table
\begin{table}[h!] \label{table-of-lemma-1-chap3}
\begin{center}
\begin{tabular}{|c|c|c|c|c|c|}
\hline 
\rule[-1ex]{0pt}{2.5ex} $r$ & $2$ & $1$ & $1/2$ & $1/3$ &  $1/4$\\ 
\hline 
\rule[-1ex]{0pt}{2.5ex} $r'$ & $-2$ &  $-1$ & $-1/2$ & $-1/3$ &  $-1/4$\\ 
\hline 
\end{tabular} 
\end{center}
%\caption{Slopes of exceptional truly cosmetic surgeries on $S^3$}
\end{table}
\end{lem}

\bigskip

\begin{proof}

Write $r=p/q$ and $r'=p/q'$. By Yi Ni and Zhongtao Wu $r=-r'$ so $q=-q'$, then $\Delta(r,r')=p|q-q'|=2 p|q| \leq 8$  i.e $p|q|\leq 4$.  Therefore $p\in \left\{1,2,3,4\right\}$.  If $p=1$ then $|q| \leq 4$ and we have the case $r\in \left\{1,1/2,1/3,1/4\right\}$. If $p=2$ then $|q| \leq 2$, since $q$ is odd we have $|q|=1$ so $r=2$. Now we need to exclude the case $p\in \left\{3,4\right\}$.

\bigskip

An orientation preserving homeomorphism  $f:M(p/q)\to M(p/q')$ induces an isomorphism $f_*:H_1(M(p/q))\to H_1(M(p/q'))$.  Since $H_1(M(p/q))=\Z/p\Z$ is generated by the class $\left[\mu \right]_q $ of the meridian,
 $$f_*\left[\mu \right]_q =\left[\mu \right]_{q'} u, \hspace{1.5cm} \text{  for some unit } u \in \left(\Z/p\Z\right)^*.$$

% A . B --->  - A . B  raha mivadika ilay orientation

Let's recall that the linking pairing of $M(p/q)$  is a non-degenerate bilinear form
$$lk_{q}: \text{Tor}( H_1(M(p/q))) \otimes \text{Tor} (H_1(M(p/q))) \to \Q/\Z,$$ 
which is defined via some intersection count.  One can check that  
$$lk_q(\left[\mu \right]_q,\left[\mu \right]_q)=-q/p.$$

 To see this, let $D$ be a meridian disk for the surgery torus such that $p\mu+q\lambda =\partial D,$ in $M(p/q)$. Since  $Y$ is a $\Z$-homology sphere $\lambda_M =\partial\Sigma$ for some surface $\Sigma$. Then  $p\mu =\partial D-q \partial \Sigma=\partial \left(D-q\Sigma\right)$  and by definition
$$lk_q(\left[\mu \right]_q,\left[\mu \right]_q)=\frac{\left(D-q\Sigma\right) \centerdot \mu}{p}\ \ \ \left[\text{mod}\ \Z\right]$$
where the dot ``$\mathord{\centerdot}$" denotes the  intersection number.
 Now $\mu$ can be pushed off of $D$ so $D \centerdot \mu =0$, and $\partial\Sigma =\lambda_M\Sigma$ so $\Sigma \centerdot \mu =1$. Therefore
$$lk_q(\left[\mu \right]_q,\left[\mu \right]_q)=-\frac{q}{p} \ \ \ \left[\text{mod}\ \Z\right].$$
%Similarly in $M(p/q')$
%$$lk_{q'}(\left[\mu \right]_{q'}u,\left[\mu \right]_{q'}u)=-\frac{q'}{p} u^2 \ \ \ \left[\text{mod}\ \Z\right].$$

The map $f$  induces an isomorphism between the linking pairing of  $M(p/q)$  and $M(p/q')$ since it preserves oriented intersection number.  Therefore

$$
\begin{aligned}
lk_q(\left[\mu \right]_q,\left[\mu \right]_q) & =lk_q(f_*\left[\mu \right]_q,f_*\left[\mu \right]_q) \ \ \ \left[\text{mod}\ \Z\right]\\
&=lk_{q'}(\left[\mu \right]_{q'}u,\left[\mu \right]_{q'}u)\ \ \ \left[\text{mod}\ \Z\right]\\
&=lk_q(\left[\mu \right]_{q'},\left[\mu \right]_{q'}) u^2\ \ \ \left[\text{mod}\ \Z\right].
\end{aligned}$$

Thus
$$-\frac{q}{p} \equiv -\frac{q'}{p} u^2 \ \ \ \left[\text{mod}\ \Z\right], \  \ \text{i.e} \ \ \ q \equiv q' u^2 \ \ \ \left[\text{mod}\ p\right].$$

We apply this congruence  to the case $p\in \left\{4,3\right\}$.  For $p=4$ (resp. $p=3$), $u\in \left\{1,3\right\}$ (resp.  $u\in \left\{1,2\right\}$). Therefore $u^2=1$ and $q\equiv q'\ \ \ \left[\text{mod }\ 4\right]$, but  $q'\in \left\{q+1,q+2\right\}$ by Lemma \ref{my-lemma-1} which is not possible. Thus $p\notin \left\{3,4\right\}$.

\end{proof}

In this lemma it is essential that int$(M)$ has a complete finite volume hyperbolic structure since the bound is on the diameter of $E(M)$. Thus the examples given in \cite{Mathieu} do not fall into these categories. 

For a hyperbolic knot $K$ in $S^3$, if we take in account the type of manifold obtained after surgery we have the following lemma.

% AZA ADINO FOANA ILAY HYPERBOLIC
\begin{lem} \label{my-lemma-2-chap3}
There are no truly cosmetic surgeries on hyperbolic knot in $S^3$ which yields a reducible manifold.
\end{lem}
\begin{proof}
If $K \subset S^3$ is a hyperbolic knot and $r,r'$ are two reducible slopes on $\nb (K)$ then $\Delta(r,r')\leq 1$ by Theorem~\ref{thm:reducible surgery}. However by Lemma \ref{my-lemma-1-chap3}, the distance between two cosmetic slopes is at least two. This is not possible.
\end{proof}

\subsubsection{Proof of Theorem~\ref{my-theorem-A}}

We can now finish the proof of Theorem \ref{my-theorem-A}. 
By Lemma \ref{my-lemma-2-chap3} and Corollary \ref{my-corollary-1-chap3} $S_K (r)$ is irreducible and non-Seifert fibred.  Therefore since the manifold is not hyperbolic it contains an essential torus. On the other hand, by Theorem \ref{denominator-2}, if $S_K (r)$ is toroidal and $r=p/q$ then $|q| \leq 2$.
Therefore using the distance table \ref{distance-table} and table of Lemma \ref{my-lemma-1-chap3}  we can deduce that: $r=2$ and $s=-2$, or $r=1$  and $s=-1$,  or $r=1/2$ and $s=-1/2$.

%% ASIO AO @ CHAPTER 1 KO IO THEOREM IO
%misy work of Zhang ihany ko eto
By Theorem \ref{non-integral:theorem 1.1}, if there is a non-integral slope on $\partial (S^3\setminus \mathcal{N}(K))$ which gives a toroidal manifold then $K$ is one of the Eudave-Mu\~noz knots $k(l , m, n, p)$ and the surgery is the corresponding half-integral surgery. Therefore there is at most one slope which can give an essential torus. Thus there is no non-integral cosmetic slope which give toroidal manifold. This  excludes the case  $r=1/2$ and $s=-1/2$.

Now we have either $r=1$ or $r=2$. If $r=2$ then $\Delta(2,-2)=4$ and Theorem \ref{toroidal:theorem 24.4} gives the complete list of all hyperbolic knots in $S^3$ with two toroidal slopes $r_1$ and $r_2$ at distance 4. Precisely, there is an integer $n$ and an homeomorphism of $S^3$ which send the triple $(K,r_1,r_2)$ to   $(L_1(n);0,4), \ \ n\neq 0,1$  or $(L_2(n);2-9n, -2-9n), \ \ n\neq 0,\pm 1$. Where $L_i(n), i=1,2$ denotes the knot obtained from the right component of the link $L_i,\  i=1,2$ (see figure below) after $1/n$ surgery on the left component

\bigskip
\hspace*{2cm}\centerline{\epsfbox{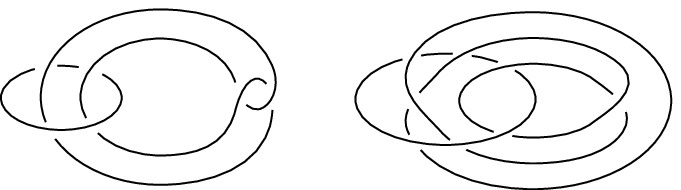}}
\bigskip
\centerline{The links $L_1$ and $L_2$ from left to right}
\bigskip

The manifold obtained after surgery is then  $S^3(L_2(n);2-9n)$ or   $S^3(L_2(n); -2-9n)$ or $S^3(L_1(n);0)$ or $S^3(L_2(n); 4)$. Therefore we can check that
$$|2-9n| = |H_1\left(S^3(L_2(n);2-9n)\right)| \neq |H_1\left(S^3(L_2(n); -2-9n)\right)| = |2+9n|$$
$$0=|H_1\left(S^3(L_1(n);0)\right)|\neq |H_1\left(S^3(L_2(n); 4)\right)| = 4.$$
Since $n\neq 0$, this completes the proof of Theorem \ref{my-theorem-A}. $\square$

\bigskip
\section{Cosmetic surgery on some special classes of knots}\label{special examples}

As  consequences of Theorem \ref{my-theorem-A}, let us give some results about cosmetic surgeries along algebraic knots, alternating knots and arborescent knots in $S^3$. 
 
 \subsection{Algebraic knots.}\label{algebraic_knot}

An \textit{algebraic knot} $K\subset S^3$ is the link of an {\em irreducible complex
plane curve singularity } that is, $K$  is the transversal intersection 
$$K=\{f=0\}\cap S^3$$
where $f:(\C^2,0)\to (\C,0)$ is an	irreducible holomorphic function, and $S^3=\{z\in \C^2:
||z||=\epsilon\}$ for $\epsilon>0$ sufficiently small. 
The natural orientations of $S^3$ and of the regular part of $\{f=0\}$
induces a natural orientation on $K$. 
% Since $f$ isirreducible, $K_f\approx S^1$.
 When $\{f=0\}$ is not smooth at the origin, $K$ is not the unknot. The Heegaard Floer homology of the result of a surgery on an algebraic knot has been computed by A. N\'emethi in \cite{Nemethi}.
 
\begin{pro}[A. N\'emethi  \cite{Nemethi}] \label{Nemethi-algebraic-knot}
Let $K \subset S^3$ be an algebraic knot,  $p,q >0  $ two coprime integers and $Y=-S^3_K(-p/q)$.  Then $\hfred(Y)$ is supported in even degree. 
\end{pro} 

This leads to the following corollary.

\begin{customcor}{\ref{my-corollary-2-chap3}}
There are no truly cosmetic surgeries on non-trivial algebraic knot in $S^3$.
\end{customcor}

\begin{proof}
The proof is similar to the proof of Corollary \ref{my-corollary-1-chap3} since the $\hfred$ of the resulting manifold is supported in even degree by Proposition \ref{Nemethi-algebraic-knot}.
\end{proof}

\subsection{Alternating knots.}\label{alternating_knot}
Combining Theorem \ref{my-theorem-A} with work of Kazuhiro Ichihara and Hidetoshi Masai, we have  the following result for alternating knot in $S^3$.

\begin{customcor}{\ref{my-corollary-3}}
There are no exceptional truly cosmetic surgeries on an alternating hyperbolic knot in $S^3$.
\end{customcor}
\begin{proof}
 
The corollary is a consequence of the classification of exceptional surgeries on alternating knots which was done by  Kazuhiro Ichihara and Hidetoshi Masai \cite{Ichihara} combined with Theorem \ref{my-theorem-A}. By this classification if an alternating hyperbolic knot in $S^3$ admits an exceptional surgery with slope $r$, then either:
\begin{itemize}
\item[•] $K$ is a twist knot  $K[2n,\pm 2]$ for $n\neq 0$ (which includes the figure-8),
\item[•] $K$ is a two bridge knot $K_{\left[a,b\right]}$ with $|a|,|b| >2$ and $r=0$ if both $a,b$ are even and $r=2b$ if $a$ is odd and $b$ is even,
\item[•] $K$ is a pretzel knot $P(a,b,c)$  with $a,b,c \neq 0,\pm 1$ and $r=0$ if $a,b,c$ are all odd and $r=2(b+c)$ if $a$ is even and $b,c$ are odd. Moreover $S^3_K(r)$ is toroidal but not Seifert fibred.
\end{itemize}
The Alexander polynomial of a  twist knot  $K[2n,\pm 2]$ is $\Delta_K(t)=2n+1 -n (t+t^{-1})$, so \\
$\Delta_K''(1)= -2n\neq 0$. Therefore by Proposition \ref{my-proposition-1} we cannot have truly cosmetic surgery for the first case.  For the last two cases $r\neq \pm 1$, so these exceptional slopes  cannot be truly cosmetic slopes by Theorem \ref{my-theorem-A}. 
\end{proof}

\bigskip

\subsection{Arborescent knot.}\label{arborescent_knot}

Another class of knots in $S^3$ is the class of \textit{arborescent knots}.   Let's recall that a  \textit{Montesinos tangle} is a sum of several rational tangles and an \textit{arborescent tangle} is one that can be obtained by summing several Montesinos tangles together in an arbitrary way. An  \textit{arborescent knot} or \textit{link} is a knot or link obtained by joining the endpoints of the arcs in an arborescent tangle by two arcs.

\begin{customcor}{\ref{my-corollary-4}}
There are no exceptional truly cosmetic surgeries on arborescent knots in $S^3$.
\end{customcor}

% arborescent links includes all 2-bridge links and Montesinos links 

\begin{proof}
In light of Theorem \ref{my-theorem-A} we only have to check that $\pm 1$ surgery on an arborescent knot do not yield two homeomorphic toroidal manifold. Following Ying-Qing Yu, there are three types of arborescent knot: type I, type II and type III. By Theorem 3.6 of \cite{Wu_Qing}, every non-trivial surgery on a type III arborescent knot gives a  hyperbolic manifold, therefore we shall focus on type II and type I arborescent knots. For type I, they are Montesinos knots with length at most 3, which again split as 2-bridge knots and Montesinos knots of length 3.  The 2-bridge knots which admit toroidal surgery are given in Theorem 1.1 of \cite{Brit-Wu} and they are among the knots in Corollary \ref{my-corollary-3}, hence they do not admit truly cosmetic surgery.	The case of Montesinos knots of length 3 is dealt with in \cite{Wu_Qing2}, precisely if $K$ is a Montesinos knot of length 3 and $\delta$ is a slope on $\partial \nb (K)$, then $S^3_K (\delta)$ is toroidal if and only if, following notation in \cite{Wu_Qing2},  $(K, \delta)$ is equivalent to one of 

\begin{itemize}

\item $K = K(1/q_1,\; 1/q_2,\; 1/q_3)$, $q_i$ odd, $|q_i|>1$,
    $\delta = 0$.

\item $K = K(1/q_1,\; 1/q_2,\; 1/q_3)$, $q_1$ even, $q_2, q_3$ odd,
 $|q_i|>1$, $\delta = 2(q_2+q_3)$.

\item $K = K(-1/2,\; 1/3,\; 1/(6+1/n))$, $n \neq 0, -1$,  $\delta = 16$
if $n$ is odd, and $0$ if $n$ is even.

\item $K = K(-1/3,\; -1/(3+1/n),\; 2/3)$, $n\neq 0, -1$, $\delta = -12$
when $n$ is odd, and $\delta=4$ when $n$ is even.

\item $K = K(-1/2,\; 1/5,\; 1/(3+1/n))$, $n$ even, and $n\neq 0$,
$\delta = 5 - 2n$.

\item $K = K(-1/2,\; 1/3,\; 1/(5+1/n))$, $n$ even, and $n\neq 0$,
$\delta = 1-2n$.

\item $K = K(-1/(2+1/n),\; 1/3,\; 1/3)$, $n$ odd, $n \neq -1$,
$\delta = 2n$.

\item $K = K(-1/2,\; 1/3,\; 1/(3+1/n))$, $n$ even, $n \neq 0$,
$\delta = 2-2n$.

\item $K = K(-1/2,\; 2/5,\; 1/9)$, $\delta = 15$.

\item $K = K(-1/2,\; 2/5,\; 1/7)$, $\delta = 12$.

\item $K = K(-1/2,\; 1/3,\; 1/7)$, $\delta = 37/2$.

\item $K = K(-2/3,\; 1/3,\; 1/4)$, $\delta = 13$.

\item $K = K(-1/3,\; 1/3,\; 1/7)$, $\delta = 1$.

\end{itemize}

 After checking this list  for knots   which are listed more than once we notice that there are at most three toroidal slopes. A knot  $K(t_1, t_2, t_3)$  admits exactly two toroidal surgeries if and only if it is equivalent to one of the following 5 knots:

$$
\begin{aligned}
& K(-1/2,\; 1/3,\; 2/11), \quad \text{$\delta = 0$ and $-3$;} \\
& K(-1/3,\; 1/3,\; 1/3), \quad \text{$\delta = 0$ and $2$;} \\
& K(-1/3,\; 1/3,\; 1/7), \quad \text{$\delta = 0$ and $1$;} \\
& K(-2/3,\; 1/3,\; 1/4), \quad \text{$\delta = 12$ and $13$;} \\
& K(-1/3,\; -2/5,\; 2/3), \quad \text{$\delta = 4$ and $6$.} 
\end{aligned}
$$

and it admits exactly three toroidal surgeries if and only if  it is the figure-8 or the knot  $ K(-1/2,\; 1/3,\; 1/7)$.  No pair of toroidal slopes corresponding to these knots admitting more that one toroidal slope are $\pm 1$ in the standard Seifert framing so this excludes the possibility of truly cosmetic surgery. 

The remaining case is that of type II arborescent knots. There are knots that have a Conway sphere cutting it into two Montesinos tangles of type $T(r_i, 1/2), \ \ i=1,2$ where $r_i \in \Q \cup \{\infty\}$. According to Theorem 1.1 of \cite{Wu_Qing3}, there are three distinct knots $K_1,K_2,K_3$ such that an arborescent knot $K\in S^3$ admits an exceptional slope $\delta$ if and only if $(K,\delta)$ is isotopic to $(K_1,3),\ (K_2,0), \ (K_3,-3) $, in which case the slope is toroidal. Therefore since there is exactly one slope for each knot, there are no truly cosmetic surgery.
\end{proof}

\bigskip

\section{Some properties of Heegaard Floer invariants.}\label{miscellaneous properties of HFH}
\subsection{The correction term}

From the absolute $\Q$-grading we can derive a new 
numerical invariant for rational homology three-spheres equipped with $\spinc$ structures.

\begin{defn}[P. Ozsv\'ath and Z. Szab\'o, \cite{Absolutely}]
Let $(Y,\s)$ be a rational homology three-sphere equipped with a $\spinc$ structure.
The  correction term $d(Y,\s)$ is the minimal $\Q$-grading of any non-torsion
element in the image of $\hfinfty(Y,\s)$ in $\hfp(Y,\s)$, i.e
$$d(Y,\s)=\min \{\liftGr\left(\pi_*(x)\right)\ \mid \ x\in \hfinfty(Y;\s) \}$$
where $\pi_*:\hfinfty(Y;\s) \to \hfp(Y;\s)$ is the map in the long exact sequence 
\begin{center}
\begin{tikzpicture}[node distance=3cm]
\node (O) {$\cdots$};
\node (A) [right of=O] {$\hfm(Y,\s)$};
\node (B) [right of=A] {$\hfinfty(Y,\s) $};
\node (C) [right of=B] {$\hfp(Y,\s) $};
\node (O1)[right of=C] {$\cdots$};
\draw[->] (O) to node {} (A);
\draw[->] (A) to node[above] {$i_*$} (B);
\draw[->] (B) to node[above] {$\pi_*$} (C);
\draw[->] (C) to node {} (O1);
\end{tikzpicture}
\end{center}
\end{defn}
There is another interpretation of  $d(Y,\s)$ using the reduced homology $\hfred(Y;\s)$. By definition of $\hfred(Y;\s)$, we have the isomorphism:
$$\hfp(Y;\s) \cong \frac{\Z[U,U^{-1}]}{U\Z[U]}\oplus\hfred(Y;\s),$$
then $d(Y,\s)$ is the $\Q$-grading of the lowest degree generator of $\Z[U,U^{-1}]/U\Z[U]$. 

For the 3-sphere, the homologies $\hfc(S^3)$ are all supported in degree zero so $d(S^3)=0$.

%The correction term satisfies the following properties.
%\begin{pro}[P. Ozsv\'ath and Z. Szab\'o, \cite{Absolutely}]
%\label{prop:CorrTermOrient}
%Let $(Y,\s)$ and $(Y',\frakt)$ be rational homology three-spheres equipped with $\spinc$ structures. Then, we have that
%\begin{eqnarray}
%& d(Y,\s)=d(Y,{\overline \s}) \\
%& d(Y,\s)=-d(-Y,\s)\\
%& d(Y\# Y',\s\#\frakt)=d(Y,\s)+d(Y',\frakt)
%\end{eqnarray}
%\end{pro}
%
%\begin{proof}
%The proof can be seen in \cite{Absolutely} section 4.
%\end{proof}
%\bigskip

For the case of a $3$-manifold $Y_0$ with $H_1(Y_0;\Z)\cong \Z$, there is a unique $\spinc$ structure $\s_0$ suth that $c_1(\s_0)=0$. We can then define two correction terms as follow.
\begin{defn}[P. Ozsv\'ath and Z. Szab\'o, \cite{Absolutely}]
We define the correction terms $d_{+1/2}(Y_0)$, resp.~$d_{-1/2}(Y_0)$, to be  the minimal $\Q$-grading of any non-torsion
element in the image of $\hfinfty(Y_0,\s_0)$ in $\hfp(Y_0,\s_0)$ with grading {\em $+1/2$  resp. $-1/2$  modulo $2$.}
\end{defn}

\begin{pro}[P. Ozsv\'ath and Z. Szab\'o, \cite{Absolutely}]
\label{pro: Correction Term for b1=1}
Let $H_1(Y_0;\Z)\cong \Z$. Then,
\begin{eqnarray}
d_{1/2}(Y_0)-1\leq d_{-1/2}(Y_0).\\ 
d_{\pm 1/2}(Y_0,\s)=d_{\pm 1/2}(Y_0,{\overline\s}),\\ 
d_{\pm 1/2}(Y_0,\s)=-d_{\mp 1/2}(-Y_0,\s). 
\end{eqnarray}
\end{pro}
\begin{proof}
See \cite{Absolutely}, section 4.2.
\end{proof}

For integral homology spheres there is only one $\spinc$ structure so there is a unique correction term. The following proposition gives a relationship between correction
terms for integral homology spheres and the correction terms for
zero-surgeries on knots they contain. 

\begin{pro}[P. Ozsv\'ath and Z. Szab\'o, \cite{Absolutely}]\label{pro:Bound on correction term}
Let $K\subset Y$ be a knot in an integral homology three-sphere and let $Y_1$ be the manifold obtained by $+1$ surgery on $K$. Then,
$$ d(Y)- \frac{1}{2} \leq d_{-1/2}(Y_0), \hspace*{0.8cm} \text{and} \hspace*{0.8cm}
 d_{+1/2}(Y_0)-\frac{1}{2}\leq d(Y_1).$$
\end{pro}
\begin{proof}
See \cite{Absolutely} section 4.2.
\end{proof}

%\subsubsection{Correction term and Knots in $S^3$}

In light of Proposition~\ref{pro:Bound on correction term}, the correction terms
for homology $S^1\times S^2$ can be used to give obstructions for
obtaining a given three-manifold as zero-surgery on a knot in the
three-sphere. Specifically, since $d(S^3)=0$, we see that if $Y_0$ is
obtained as zero-surgery on a knot in $S^3$, then $$-\frac{1}{2}\leq
d_{-1/2}(Y_0).$$ Moreover, by reflecting the knot  %and usingProposition~\ref{prop:CorrTermOrient}, 
we also obtain the bound
$$d_{1/2}(Y_0) \leq \frac{1}{2}.$$

\begin{pro}[P. Ozsv\'ath and Z. Szab\'o, \cite{Absolutely}]
\label{pro:Correction term for S3}
Let $K\subset S^3$ be an oriented knot in the three-sphere. Then, 
\begin{eqnarray*}
d_{1/2}(S^3_{K}(0))-\frac{1}{2} &=& d(S^3_K(1)) \\
d(S^3_{K}(-1))-\frac{1}{2} &=& d_{-1/2}(S^3_{K}(0))
\end{eqnarray*}
\end{pro}

\begin{proof}
This is a direct consequence of the 
surgery long exact sequence, in view of the structure of
$\hfp(S^3)$.
\end{proof}

 Ozsv\'ath and Z. Szab\'o also proved the following result concerning the correction term for $1/n$-surgery. 
\begin{pro}[P. Ozsv\'ath and Z. Szab\'o, \cite{Absolutely}]
\label{pro:fractional-sugery}
Let $K\subset Y$ be a knot in an integral homology three-sphere.
Then, we have the following inequalities 
(where here $n$ is any positive integer):
$$d_{1/2}(Y_K(0))-\frac{1}{2}\leq  d(Y_K(1/(n+1)))\leq d(Y_K(1/{n})) \leq d(Y) $$
$$d(Y)\leq d(Y_K(-1/n))\leq d(Y_K(-1/(n+1))) 
\leq d_{-1/2}(Y_K(0))+\frac{1}{2}.$$
\end{pro}
\begin{proof}
See \cite{Absolutely} Corollary 9.14.
\end{proof}

\subsection{Torsion invariant of the Alexander polynomial}
Here are some results relating the torsion of the Alexander polynomial of a knot in an integer homology sphere to $+1$-surgery. In what follows, $Y$ will be an oriented integer homology sphere. We define
$$N(Y):=\text{rank} \hfred(Y)$$
\begin{defn}
Let $K$ be a knot in a rational homology sphere $Y$ and let its normalized Alexander polynomial be
$$\Delta_K(T)=a_0+\sum_{i=1}^d\ a_i\ (T^i+T^{-i}).$$
We define the $i$-th torsion invariant of the Alexander polynomial to be 
$$t_i=\sum_{j=1}^d\ ja_{|i|+j}.$$
\end{defn}

\begin{thm}[P. Ozsv\'ath and Z. Szab\'o, \cite{Absolutely}]\label{thm: bound on torsion invariant of Alex Poly}
Let $Y$ be an integral homology three-sphere and $K\subset Y$ be a knot. Then
there is a bound:
$$|t_0(K)| + 2 \sum_{i=1}^d |t_i(K)|\leq \textnormal{\rank}  \hfred(Y)+\frac{d(Y)}{2}+ \textnormal{\rank}  \hfred(Y_K(1))-\frac{d(Y_K(1))}{2},$$
\end{thm}
\begin{proof}
See \cite{Absolutely} Theorem 6.1.
\end{proof}

\subsection{Correction term and torsion invariants for exceptional cosmetic surgeries.}

Theorem~\ref{my-theorem-A} also induces the following results about correction terms.

\begin{cor}\label{my-corollary-A1}
If a 3-manifold $Y$ is the result of an exceptional truly cosmetic surgery on a hyperbolic knot $K$ in $S^3$ then:
$$|t_0(K)| + 2 \sum_{i=1}^n |t_i(K)|\leq \rank  \hfred(Y),$$
where the number $t_i(K)$ for $i\in \Z$ is the torsion invariant of the Alexander polynomial  $\Delta_K(T)$ of $K$ and $n$ is the degree of $\Delta_K(T)$.
\end{cor}

This lower bound is strictly positive since not all the torsion $t_i(K)$  are zero.

\begin{proof}
Let $K\subset S^3$ be a hyperbolic knot such that there is an orientation preserving homeomorphism between $S^3_K(r)$ and  $S^3_K(r')$ for two distinct rational numbers $r$ and $r'$. Let $Y=S^3_K(r)$, by Theorem \ref{my-theorem-A} we can assume $r=+1$. By Theorem \ref{thm: bound on torsion invariant of Alex Poly} we have the inequality:
$$|t_0(K)| + 2 \sum_{i=1}^n |t_i(K)|\leq \rank  \hfred(S^3)+\frac{d(S^3)}{2}+\rank  \hfred(Y)-\frac{d(Y)}{2},$$
where the number $t_i(K)$ for $i\in \Z$ is the torsion invariant of the Alexander polynomial  $\Delta_K(T)$ of $K$ and $n$ is the degree of $\Delta_K(T)$. On the other hand we also have the identity:
$$\lambda(Y)= \chi (\hfred(Y))-\frac{d(Y)}{2}$$
%from Theorem~\ref{theorem:renormalized-euler-char}. We also know that $\rank  \hfred(S^3) 
from \cite{Rustamov}, where $\lambda$ stand for the Casson invariant. We also know that $\rank  \hfred(S^3)= d(S^3)=0$. Now by the surgery formula for Casson invariant 
$$\lambda(Y)=\lambda(S^3)+\lambda(L(1,1))+ \Delta''_K(1) = \Delta''_K(1) $$
and by Proposition \ref{my-proposition-1} $\Delta''_K(1) =0$, thus $\lambda (Y)=0$. By Proposition \ref{Ni-Wu-result} $\chi (\hfred(Y))=0$, hence $d(Y)=0$. This proves the desired result.  
\end{proof}
%In particular we retrieve that $Y$ is not an $L$-space since the $t_i$ cannot all be zero.

\begin{cor}\label{my-corollary-A2}
If a hyperbolic knot $K\subset S^3$ admits an exceptional truly cosmetic surgery then the Heegaard Floer correction term of any $1/n$ ($n\in \Z$) surgery on $K$ satisfies $$d(S^3_K(1/n))=0.$$
\end{cor}
\begin{proof}
Let $K$ be as in proof of Corollary \ref{my-corollary-A1}.
Let  $d_{1/2}(S^3_K(0))$ and  $d_{-1/2}(S^3_K(0))$ be the two correction terms for the $0$-surgery along $K$.  Let $n$ be a positive integer, by Proposition~\ref{pro:fractional-sugery}, 
$$d_{1/2}(S^3_K(0))-\frac{1}{2} \leq d(S^3_K(1/(n+1))) \leq d(S^3_K(1/n))\leq d(S^3)=0$$
By Proposition \ref{pro:Correction term for S3} we have
$$d_{1/2}(S^3_K(0))-\frac{1}{2}= d(S^3_K(+1)).$$ %, \ \ \ \ \text{and} \ \ \ d_{-1/2}(S^3_K(0))+\frac{1}{2}= d(S^3_K(-1))$$
By the proof of Corollary \ref{my-corollary-A1}  $d(S^3_K(+1))=0$, this completes the proof.
\end{proof}
% note
% Bib files are VERY PICKY about syntax, more so than python
% Don't mix your Ps and Qs or commas and brackets.

%\begin{thebibliography}{9999} \label{reference}
%
%\end{thebibliography}

\medskip\noindent
\vspace{.5cm} 

\scriptsize{
\noindent
Huygens C. Ravelomanana, CIRGET-UQAM 
\newline\noindent
P. O. Box 8888, Centre-ville, Montr\'eal, Qc, H3C 3P8, Canada
\newline\noindent
e-mail: huygens@cirget.ca  
}

\end{document}